\documentclass[10pt]{amsart}

\usepackage{graphicx}
\usepackage{amsmath}
\usepackage{amssymb}
\usepackage{mathrsfs} 
\usepackage{xcolor}

\usepackage[utf8]{inputenc}

\newtheorem{thm}{Theorem}[section]
\newtheorem{prop}[thm]{Proposition}
\newtheorem{lem}[thm]{Lemma}
\newtheorem{cor}[thm]{Corollary}

\newtheorem{rem}[thm]{Remark}

\theoremstyle{definition}
\newtheorem{definition}[thm]{Definition}

\theoremstyle{remark}

\newtheorem{remark}[thm]{Remark}

\numberwithin{equation}{section}

\newcommand{\R}{\mathbb{R}} 
\newcommand{\C}{\mathbb{C}} 
\newcommand{\Z}{\mathbb{Z}}

\newcommand{\T}{\mathbb{T}}
\newcommand{\TT}{\mathcal{T}}

\newcommand{\CC}{\mathrm{C}}
\newcommand{\dH}{\dim_{\mathrm{H}}}
\DeclareMathOperator{\M}{M}

\begin{document}

\title[Hausdorff dimension of measures with arithmetically restricted spectrum]{Hausdorff dimension of measures with arithmetically restricted spectrum}

\author{R. Ayoush}
\address{Institute of Mathematics, Polish Academy of Sciences
00-656 Warszawa, Poland}
\email{rayoush@impan.pl}

\author{D. Stolyarov}
\thanks{D. S. is supported by the Russian Science Foundation grant N. 19-71-30002.}
\address{Department of Mathematics and Computer Science, St. Petersburg State University; 199178,14th line 29, Vasilyevsky Island, St. Petersburg, Russia}
\email{d.m.stolyarov@spbu.ru}

\author{M. Wojciechowski}
\address{Institute of Mathematics, Polish Academy of Sciences
00-656 Warszawa, Poland}
\email{miwoj.impan@gmail.com}


\begin{abstract}
We provide an estimate from below for the lower Hausdorff dimension of measures on the unit circle based on the arithmetic properties of their spectra. 
We obtain our bounds via application of a general result for abstract $q$-regular martingales to the Gundy--Varopoulos backwards martingale.
To show the sharpness of our method, we improve the best known numerical lower bound for the Hausdorff dimension of certain Riesz products. 
\end{abstract}

\maketitle

\section{Introduction}
The most common way to estimate the lower Hausdorff dimension of a measure using Harmonic Analysis tools is the so-called energy method. It involves examination of the summability properties of the Fourier coefficients of a measure. In general, however, the energy and Hausdorff dimensions may be different (see, e.g. Proposition 3.4 in~\cite{HR} or Chapter 13 in~\cite{M}). In this paper, we investigate not only the size of the spectrum, but also its arithmetic properties.

By $\T = \R / \Z$ we denote the circle group. 
\begin{definition}
	Let~$\mu$ be a finite (non-negative) Borel measure on~$\T$. The quantity
	\begin{equation*}
	\dH(\mu) = \inf \{\alpha\colon \hbox{there exists a Borel set~$F$ such that} \ \mu(F) \neq 0, \ \dH F \leq \alpha \}
	\end{equation*}
is called the lower Hausdorff dimension of $\mu$.
\end{definition}


\begin{definition}
	Let $A \subset \Z$. We denote by $\M_{A}(\T)$ the set of finite Borel measures satisfying $\hat{\mu}(n) = 0$ for any $n \in \Z \setminus A$.
\end{definition}

Throughout the article $q$ is a fixed integer greater than $2$.  The symbol $ \parallel $ means the relation of exact division of integers. That is $a^{n} \parallel b$ if and only if $a^{n} | b$ but $a^{n+1} \nmid b$. For any $B \subset \{1, 2, \dots , q-1\}$, let us define 
\begin{equation*}
\CC_{B} = \{kq^{n}\colon k \mod q \in B,\ n \geq 0\}\cup\{0\}.
\end{equation*}
We denote  the group of residues modulo~$q$ by~$\mathbb{Z}_q$ and identify the set~$\{0,1,\ldots,q-1\}$ with it in the natural way. 
Our first result may be thought of as an uncertainty principle (see~\cite{HJ}). 
\begin{thm} \label{UP}
	Let $B\subset \mathbb{Z}_q\setminus\{0\}$ and let $\mu \in \M_{\CC_{B}}(\T)$ be a finite non-negative measure. If $B \subset H\setminus\{0\}$ for some subgroup $H\subset \mathbb{Z}_q$\textup, then
	\begin{equation*}
\dH(\mu) \geq 1 - \frac{\log |H|}{\log q}.
	\end{equation*}
	Moreover, if the inclusion~$B \subset H \setminus \{0\}$ is proper, then the above inequality is strict in the following sense\textup: there exists~$\delta > 0$ independent of~$\mu$ such that
	\begin{equation*}
\dH(\mu) \geq 1 - \frac{\log |H|}{\log q} + \delta.
	\end{equation*}
	In particular\textup, if $B \neq \mathbb{Z}_q\setminus\{0\}$\textup, then~$\dH(\mu) > \delta$ for any non-negative~$\mu \in \M_{C_{B}}(\T)$.
\end{thm}
This theorem is a corollary of more general Theorem~\ref{transfer} below. The latter theorem provides better bounds based on the arithmetic structure of the set~$B$. In particular, it delivers simple numeric bounds for~$\delta$ in Theorem~\ref{UP}. However, Theorem~\ref{transfer} requires more notation, so we leave its formulation for a while.

We confront our methods with the question about determining the dimension of Riesz products. For convenience, let us focus on the class given by
\begin{equation}
	\mu_{a,q} = \prod_{k=0}^{\infty}\big(1+a\cos(2\pi q^{k}x)\big),
\end{equation}
where $a \in [-1,1]$. One of the most important advances in the mentioned problem is contained in the seminal work \cite{Pey} of Peyrière. In this paper, among other things, he proved the identity
\begin{equation}\label{peyr}
	\dH(\mu_{a,q}) = 1 - \frac{\int_{0}^{1}\log(1+a\cos(2\pi x))d\mu_{a,q}}{\log q}.
\end{equation}
We note that Peyrière considered Riesz products of more general type. Results of his work go beyond Hausdorff dimension estimates and shed light on random nature of those measures. Later, Fan~\cite{F2} gave an approximation result using probabilistic methods
\begin{equation}\label{fanthm}
\dH(\mu_{a,q}) = 1 - \frac{1}{\log q}\int\limits_{0}^{1}\log\big(1+a\cos(2\pi x)\big)\big(1+a\cos(2\pi x)\big)\,dx + O\Big(\frac{a}{q^2\log q}\Big),
\end{equation}
when~$q$ is a large number and~$|a| \leq \cos\big(\frac{\pi}{\lfloor\frac{q+1}{2}\rfloor+1}\big)$.

In contrast to the above, we are mainly interested in the case of (heuristically) the most singular Riesz products, i.e when $|a|$ is close or equal to 1. For such parameters we improve the best numerical lower bounds for $\dH(\mu_{a,q})$ derived directly from formula~\eqref{peyr} and those obtained by potential-theoretic methods (see~\cite{HR}, Corollary~3.2. and~\cite{M}, Corollary~13.4). The following theorem is a corollary of the already mentioned Theorem~\ref{transfer} below.
\begin{thm}\label{rieszestimate1}
For any integer $q\geq 3$ and $a \in [-1,1]$\textup, we have
\begin{equation*}
	\dH(\mu_{a,q}) \geq 1 - \frac{1}{q \log q}\sum\limits_{j=1}^{q-2}\Big(1 - \frac{\cos\frac{(2j+1)\pi}{q}}{\cos\frac{\pi}{q}}\Big)\log \Big(1 - \frac{\cos\frac{(2j+1)\pi}{q}}{\cos\frac{\pi}{q}}\Big)
\end{equation*}
\end{thm}

 Theorem~\ref{rieszestimate1} delivers bounds which may be thought of as extensions of~\eqref{fanthm}.

\begin{prop}\label{rieszestimate}
	For any even integer $q\geq 4$ and $a \in [-1,1]$\textup, we have
	\begin{multline*}
	\dH(\mu_{a,q}) \geq\\ 1 - \frac{1-\log 2}{\log q} - \frac{1}{q\log q}\bigg(2\log 2 - \frac{2}{q\cos\frac{\pi}{q}}\int\limits_{\frac{\pi}{2}}^{\frac{q\pi}{4}}\log(\cos^2 z)\sin\frac{2z}{q}\,dz\bigg) - \frac{\log \cos\frac{\pi}{q}}{\log q}.
	\end{multline*}
\end{prop}

\begin{prop}\label{RieszEstimate}
For any integer $q\geq 3$ and $a \in [-1,1]$\textup, we have
\begin{equation*}
	\dH(\mu_{a,q}) \geq 1 - \frac{1-\log 2}{\log q} - \frac{4\pi}{q\log q} - \frac{1}{\log q}\Big(\frac{1}{\cos\frac{\pi}{q}} - 1\Big).
\end{equation*}
\end{prop}
By virtue of the identity $\int_0^{1}(1+\cos 2\pi x)\log(1+\cos 2\pi x)\,dx = 1-\log 2$, the above expressions agree with the right hand side of~\eqref{fanthm} up to asymptotically the most significant terms.

Our methods are quite different from that of~\cite{F1},~\cite{F2},~\cite{K}, and~\cite{Pey}; the proofs presented here are self-contained. In particular, we do not use any sort of an ergodic theorem. We adjust the methods for estimating the lower Hausdorff dimension of the so-called Sobolev martingales from \cite{ASW}. Those martingales are vector valued. The reasoning simplifies significantly in the present case of non-negative scalar measures. More specifically, we will relate the Gundy--Varopoulos backwards martingale to a measure~$\mu \in \M_{\CC_B}$ and extract the estimate for~$\dH(\mu)$ from the growth bounds for the corresponding martingale.


\section{Transference of results from martingale spaces}
We will be representing the points of $\T$ in the $q$-ary system. We denote by~$x(j)$ the $j$-th digit of $x \in \T$, that is,
\begin{equation*}
x = \sum_{j=1}^{\infty} \frac{x(j)}{q^{j}}, \ \ x(j) \in \{0,1,2,\dots,q-1\},
\end{equation*}
with the convention that if there are two such representations, then we choose the finite one.
\subsection{Approximating trees and the Gundy--Varopoulos backwards martingale} 
Before we give precise formulas for the Gundy--Varopoulos martingale, let us briefly discuss our strategy.

Our purpose is to define, for any natural $N$, a tree $\TT_{N}$ that will be used to sample measures up to the scale $\sim q^{-N}$. Namely, the root of the tree will encode $\T$, the set of leaves will represent the arcs of length $\sim q^{-N}$, and the intermediate vertices will correspond to some periodic sets. This discretization procedure will allow us to obtain a bound for martingale approximations of a given measure (Lemma \ref{Frostman} below), depending on certain space of admissible martingale differences (which is computable in terms of Fourier coefficients, c.f. Lemma \ref{dftlemma} below). The obtained inequality will allow us to use a Frostman-type Lemma~$2.4$ from~\cite{SW}. Unfortunately, we cannot simply refer to that lemma, so we adjust its proof to our case; in fact, the proof of Theorem~\ref{transfer} presented at the end of this section follows the lines of the proof of the said lemma.

\begin{definition}
Let us introduce the set 
\begin{equation*}
\alpha_{N;\emptyset} = \{x\in \T\colon x(j)=0 \ \text{for} \ j>N \}.
\end{equation*}
For any sequence $(i_{1},\dots, i_{k})$ with~$k \leq N$ and~$i_j\in \{0,1,\dots,q-1\}$ for $j=1,2,\ldots,k$, we also introduce the set
\begin{equation*}
\alpha_{N;i_{1},i_{2},\dots, i_{k}} = \{x \in \alpha_{N;\emptyset}\colon x(N-j+1) = i_{j} \ \text{for all } j=1,2,\ldots,k \}.
\end{equation*}
\end{definition}

The above sets will be the vertices of the tree~$\TT_{N}$ described in the forthcoming definition. This tree will be regular (each parent has~$q$ children) and moreover, the sons of a parent will be enumerated by numbers from~$0$ to~$q-1$.

\begin{definition}
	We define the tree~$\TT_{N}$ according to the following rules:
	\begin{enumerate}
	\item the root of~$\TT_N$ is the set~$\{\alpha_{N;\emptyset}\}$,
	\item the $j$-th child of the root is $\alpha_{N;j}$, here $j=0,\dots, q-1$,
	\item the $j$-th child of the vertex corresponding to~$\alpha_{N;i_{1},\dots,i_{k-1}}$ is~$\alpha_{N;i_{1},\dots,i_{k-1},j}$, here $j=0,\dots, q-1$.
    \end{enumerate}
	For a vertex $\alpha$, we denote its $j$-th child by $\alpha[j]$. Let us call the set of vertices whose distance from the root is exactly~$k$ by~$\TT_{k,N}$, where~$0\leq k \leq N$.
	
\end{definition}
Note that $\TT_{N}$ is a $q$-regular tree of heigth $N$ such that the elements of $\TT_{k,N}$ are~$q^{k-N}$-periodic subsets of $\T$.

We recollect some basic facts about the Gundy--Varopoulos martingales (see~\cite{CJ} and~\cite{GV}). Consider the discrete probability space~$(\alpha_{N;\emptyset}, 2^{\alpha_{N;\emptyset}}, \nu_{N})$,
where~$\nu_N$ is the uniform probability measure on~$\alpha_{N,\emptyset}$:
\begin{equation}\label{CountingMeasure}
\nu_N = \frac{1}{q^{N}} \sum_{j=0}^{q^{N}-1} \delta_{\frac{j}{q^{N}}}.
 \end{equation}
Pick a function $f \in C(\T)$ and define 
\begin{equation}\label{martdef}
	f_{k}(x) = \frac{1}{q^{N-k}} \sum^{q^{N-k}-1}_{j=0} f\Big(x + 
	\frac{j}{q^{N-k}}\Big),\quad k=0,1 \dots, N, \ x \in \alpha_{N,\emptyset}.
\end{equation}
We restrict our attention to~$x\in \alpha_{N,\emptyset}$ only, even though the previous formula makes sense for arbitrary~$x\in \T$. The function~$f_{k}$ is~$q^{k-N}$ periodic, so, it is constant on each of the sets corresponding to the vertices in~$\TT_{k,N}$. That means we can identify~$f_k$ with a function on~$\TT_{k,N}$. One may verify that the sequence~$f_0,f_1,\ldots,f_N$
is a martingale with respect to the filtration~$\{\sigma(\TT_{k,N})\}_{k=0}^N$, where~$\sigma(T_{k,N})$ is the algebra of all~$q^{k-N}$ periodic subsets of~$\alpha_{N,\emptyset}$. Note that the elements of~$\TT_{k,N}$ are the atoms of~$\sigma(\TT_{k,N})$.

We may express the~$f_k$ in Fourier terms:
\begin{multline}\label{mart}
	f_{k}(x) = \frac{1}{q^{N-k}}\sum_{j=0}^{q^{N-k}-1}\sum_{l \in \Z} \hat{f}(l)e^{2\pi il(x+\frac{j}{q^{N-k}})}
\\
	=\sum_{l \in \Z} \Big{(} \hat{f}(l)e^{2\pi ilx} \cdot \frac{1}{q^{N-k}}\sum_{j=0}^{q^{N-k}-1} e^{2\pi i \frac{lj}{q^{N-k}}} \Big{)}
\\
	= \sum_{q^{N-k} | l} \hat{f}(l)e^{2\pi i lx},
\end{multline}
for any $x \in \alpha_{N;\emptyset}$ (this relation also holds true for any~$x\in \T$). Hence, the $k$-th martingale difference may be expressed as
\begin{equation} \label{martdif}
df_{k}(x)= f_{k}(x) - f_{k-1}(x) = \sum_{q^{N-k} \parallel l} \hat{f}(l)e^{2\pi i lx},\quad x\in \alpha_{N;\emptyset}.
\end{equation}

We use the notation
\begin{equation*}
\R^{q}_{0} = \Big\{(x_{1},\dots,x_{q}) \in \R^{q}\colon \sum_{j=1}^{q} x_{j} = 0\Big\}
\end{equation*}
and identify vectors $x \in \R^{q}$ with functions on $\Z_{q}$ in the natural way.


\begin{lem}\label{dftlemma}
For any $\alpha \in \TT_{k-1,N}$ we have
\begin{multline*}
\Big(df_{k}(\alpha[0]),df_{k}(\alpha[1]),\dots,df_{k}(\alpha[q-1])\Big) = \\
\sum_{m=1}^{q-1} \Big{(} \sum_{n \in \Z} \hat{f}\big((m+nq)q^{N-k}\big)e^{2\pi i(m+nq)q^{N-k}x_{0}} \Big{)} \omega_{m},
\end{multline*}
where $x_{0}\in \alpha$ and 
\begin{equation*}
\omega_{m} :=(\omega^{mj})_{j=0}^{q-1} :=\Big(e^{\frac{2\pi i m j}{q}}\Big)_{j=0}^{q-1},\quad j=0,1,\ldots,q-1,
\end{equation*}
are the rows of the inverse~$q\times q$ Fourier matrix. 
\end{lem}
\begin{remark}\label{FourierZqRemark}
In other words, the vector~$(df_{k}(\alpha[0]),df_{k}(\alpha[1]),\dots,df_{k}(\alpha[q-1]))$ is a properly normalized inverse $\Z_{q}$-Fourier transform of the vector
$(e_{0},e_{1},\dots,e_{q-1})$ with~$e_{0} = 0$ and
\begin{equation*}
e_{m} = \sum_{n \in \Z} \hat{f}((m+nq)q^{N-k})e^{2\pi i(m+nq)q^{N-k}x_{0}},\quad m =1,2,\ldots, q-1.
\end{equation*}
\end{remark}
The above lemma is standard, see, e.g.~\cite{CJ}. We provide its proof for completeness.
\begin{proof}[Proof of Lemma~\ref{dftlemma}.]
Let us prove our formula for each coordinate individually. For any~$j$, $j=0,1,\ldots,q-1$, we would like to show
\begin{equation*}
df_k(\alpha[j]) = \sum\limits_{m=1}^{q-1}\sum_{n\in\mathbb{Z}} \hat{f}\big((m+nq)q^{N-k}\big)e^{2\pi i (m+nq)q^{N-k}x_0}e^{\frac{2\pi i mj}{q}}.
\end{equation*}
Note that this expression does not depend on~$x_0\in \alpha$ since~$q^{N-k}(x_0 - x_0')\in\mathbb{Z}$ for any other~$x_0' \in\alpha$. On the other hand, we may use~\eqref{martdif} by representing~$x\in \alpha[j]$ as~$x=x_0 + \frac{j}{q^{N-k+1}}$, where~$x_0\in \alpha$:
\begin{multline*}
df_{k}(x)= \sum_{q^{N-k} \parallel l} \hat{f}(l)e^{2\pi i lx} = \sum\limits_{m=1}^{q-1}\sum\limits_{n\in\mathbb{Z}}\hat{f}\big((m+nq)q^{N-k}\big)e^{2\pi i (m+nq)q^{N-k}x} =\\ \sum\limits_{m=1}^{q-1}\sum\limits_{n\in\mathbb{Z}}\hat{f}\big((m+nq)q^{N-k}\big)e^{2\pi i (m+nq)(x_0 + \frac{j}{q^{N-k+1}})q^{N-k}} =\\ \sum\limits_{m=1}^{q-1}\sum_{n\in\mathbb{Z}} \hat{f}\big((m+nq)q^{N-k}\big)e^{2\pi i (m+nq)q^{N-k}x_0}e^{\frac{2\pi i mj}{q}}.
\end{multline*}
\end{proof}

\begin{definition}
Let~$W_B$ be the linear subspace of~$\mathbb{R}_0^q$ consisting of vectors~$d$ whose~$\mathbb{Z}_q$ Fourier transform vanishes on~$B$:
\begin{equation*}
W_B = \Big\{d\in \mathbb{R}_0^q\colon \forall m \in B\quad \sum\limits_{j=0}^{q-1}e^{-\frac{2\pi i mj}{q}}d_j = 0\Big\}.
\end{equation*}
\end{definition}
\begin{lem}\label{DifferencesInWB}
Let~$f\in C(\T)$ be such that~$f\,dx\in \M_{\CC_B}$. For any~$\alpha \in \TT_N$, we have the inclusion
\begin{equation*}
\Big(df_{k}(\alpha[0]),df_{k}(\alpha[1]),\dots,df_{k}(\alpha[q-1])\Big)\in W_B.
\end{equation*} 
\end{lem}
\begin{proof}
In view of Remark~\ref{FourierZqRemark},~$e_m = 0$ for any~$m\in B$ in the terminology of that remark, provided~$f\,dx\in \M_{\CC_B}$.
\end{proof}
\subsection{A general dimension estimate}
Consider an auxillary function $\kappa\colon \mathbb{R}_+\to \mathbb{R}$ defined by the rule
\begin{equation}\label{kappa}
\kappa(\theta) = \sup\Big\{\theta\log\Big(\frac{1}{q}\sum\limits_{j=1}^q|1+v_j|^{\frac{1}{\theta}}\Big)\colon  v \in  W_B  \hbox{ and } \forall j \quad v_j \geq -1\Big\}.
\end{equation}
One may verify that~$\kappa$ is continuous and convex, and therefore, has the left derivative at~$1$. Using this, we may compute the value
\begin{equation} \label{kappaprim}
\kappa'(1) = \inf\Big\{-\frac{1}{q}\sum\limits_{j=1}^q(1+v_j)\log(1+v_{j})\colon  v \in W_B \hbox{ and } \forall j \quad v_j \geq -1\Big\},
\end{equation}
where the derivative here means the left derivative.
The next lemma is simply a reformulation of the definition of~$\kappa$.
\begin{lem}\label{DefReform}
For any $a\geq0$ and any vector~$b = (b_{i})_i \in W_B$ such that $b_{j} \geq -a$ for any~$j =0,1,\ldots,q-1$, we have
\begin{equation*}
\bigg{(} \frac{1}{q} \sum^{q}_{j=1} |a+b_{j}|^{p} \bigg{)}^{\frac{1}{p}} \leq a e^{\kappa(p^{-1})}.
\end{equation*}
\end{lem}

Our main tool is the following principle established in \cite{ASW} and adjusted to our case.
\begin{thm}\label{transfer}For any finite non-negative measure $\mu \in \M_{\CC_B}$, we have
\begin{equation*}
	\dH(\mu) \geq 1 + \frac{\kappa'(1)}{\log q}.
\end{equation*}
\end{thm}

Let $\{\Phi_{N}\}_{N\geq1}$ be a non-negative and smooth approximate identity with the following properties:
		\begin{equation*}
		\Phi_{N}(x) = \begin{cases}
		q^{N}               & \text{on} \ \ [-\frac{1}{2q^{N}},\frac{1}{2q^{N}}];\\
		\leq q^{N}              &  \text{on} \ \ [-\frac{1}{2q^{N-1}},\frac{1}{2q^{N-1}}]\setminus [-\frac{1}{2q^{N}},\frac{1}{2q^{N}}];\\
		0 & \text{otherwise}.
		\end{cases}
		\end{equation*}

Observe that
\begin{equation}\label{app}
\mu\Big(\Big[x-\frac{1}{2q^{N}},x+\frac{1}{2q^{N}}\Big]\Big) \leq \frac{1}{q^{N}}\Phi_{N}*\mu(x) \leq \mu\Big(\Big[x-\frac{1}{2q^{N-1}},x+\frac{1}{2q^{N-1}}\Big]\Big)
\end{equation}
for any~$x\in \T$, in particular, for~$x \in \alpha_{N;\emptyset}$.
The inequalities~\eqref{app} establish a relationship between metric measure structures on $\TT_{N}$ and $\T$. Henceforth, we will be using results concerning the backwards martingale generated by the continuous function $f = \Phi_{N}*\mu$. Note that~$f\,dx \in \M_{\CC_B}$ provided~$\mu \in \M_{\CC_B}$.
\begin{lem}\label{TwoInequalitiesLemma}
Consider the martingale~$\{f_k\}_{k=0}^N$ generated by~$f = \Phi_{N}*\mu$ via formula~\eqref{martdef}. If~$\mu \in \M_{\CC_B}(\T)$, then
\begin{equation}\label{TwoInequalities}
\lVert f \rVert_{L_{p}(\nu_{N})} \leq e^{\kappa(p^{-1})N} \lVert f_{0} \rVert_{L_{p}(\nu_{N})} \leq q e^{\kappa(p^{-1})N} \lVert \mu \rVert.
\end{equation}
\end{lem}
We recall that~$\nu_N$ is the counting measure defined in~\eqref{CountingMeasure}. 
\begin{proof}
Let us prove the first inequality in~\eqref{TwoInequalities}. This inequality will follow provided we justify the single step bound
\begin{equation*}
\|f_k\|_{L_p(\nu_N)} \leq e^{\kappa(p^{-1})}\|f_{k-1}\|_{L_p(\nu_N)}
\end{equation*}
for any~$k=1,2\ldots, N$. This inequality, in its turn, follows from even more localized ones: for any~$\alpha \in \TT_{k-1,N}$, we have
\begin{equation*}
\Big(\sum\limits_{x\in \alpha} |f_k(x)|^p\Big)^{\frac{1}{p}} \leq e^{\kappa(p^{-1})}\Big(\sum\limits_{x\in \alpha} |f_{k-1}(x)|^p\Big)^{\frac{1}{p}}.
\end{equation*}
To prove this inequality, we note that since $\mu \geq 0$, the sequence~$\{f_k\}_k$ consists of non-negative functions. What is more,~$f_k = f_{k-1} + df_k$ and the vector
\begin{equation*}
df_k|_{\alpha} = \big(df_{k}(\alpha[0]),df_{k}(\alpha[1]),\dots,df_{k}(\alpha[q-1])\big)
\end{equation*}
lies in~$W_B$ by Lemma~\ref{DifferencesInWB}. So, the desired inequality is proved by application of Lemma~\ref{DefReform} with~$a = f_{k-1}(\alpha)$ and~$b = df_k|_{\alpha}$.

To prove the second inequality in~\eqref{TwoInequalities}, we use that $f_{0} \equiv \frac{1}{q^{N}} \sum_{x \in \TT_{N,N}} \Phi_{N}*\mu(x)$ on~$\alpha_{N;\emptyset}$:

\begin{multline*}
\lVert f_{0} \rVert_{L_{p}(\nu_{N})} = \frac{1}{q^{N}} \sum_{x \in \TT_{N,N}} \Phi_{N}*\mu(x) \stackrel{\scriptscriptstyle\eqref{app}}{\leq}\\ \sum_{x \in \TT_{N,N}} \mu\Big(\Big[x-\frac{1}{2q^{N-1}},x+\frac{1}{2q^{N-1}}\Big]\Big) \leq q \lVert \mu \rVert.
\end{multline*}
\end{proof}

\begin{lem}\label{Frostman}
For any  any $\beta < 1+\frac{\kappa'(1)}{\log q}$\textup, there  exists $\gamma$ such that
\begin{equation}\label{SetAverageEstimate}
	\frac{1}{q^{N}}\sum_{x \in C} f(x) \lesssim  \big{(} \# C\ q^{-\beta N} \big{)}^{\gamma} \|\mu\|
\end{equation}
for any $C \subset \alpha_{N;\emptyset}$\textup, with the constant independent of $N$.
\end{lem}

\begin{proof}
Let~$p\in (1,\infty)$ be a real to be chosen later. By Hölder's inequality and Lemma~\ref{TwoInequalitiesLemma}, we obtain
\begin{multline}
\frac{1}{q^{N}}\sum_{x \in C} f(x) \leq \lVert f \rVert_{L_{p}(\nu_{N})} \lVert \chi_{C} \rVert_{L_{p'}(\nu_{N})} = \lVert f \rVert_{L_{p}(\nu_{N})} (q^{-N} \# C)^{\frac{p-1}{p}} 
\lesssim\\ e^{\kappa(p^{-1})N} q^{-\frac{p-1}{p}N}(\# C)^{\frac{p-1}{p}}\|\mu\|  = e^{\kappa(p^{-1})N}q^{\frac{p-1}{p}(\beta-1)N}(q^{-\beta N} \#C)^{\frac{p-1}{p}}\|\mu\|.
\end{multline}
Hence \eqref{SetAverageEstimate} is true with~$\gamma = \frac{p-1}{p}$ when $e^{\kappa(p^{-1})}q^{\frac{p-1}{p}(\beta-1)}<1$, that is if
\begin{equation*}
\kappa(p^{-1})+(\beta-1)\frac{p-1}{p} \log q < 0.
\end{equation*}
This holds true when $(\beta-1)\log q < \kappa'(1)$ and $p$ is sufficiently close to $1$.
\end{proof}
As we have already said, the reasoning presented below is very much similar to the proof of Lemma~$2.4$ in~\cite{SW}.
\begin{proof}[Proof of Theorem~\ref{transfer}]
Assume the contrary: there exists a Borel set~$F$ such that
\begin{equation*}
\dH(F) < \beta_{1} < 1+\frac{\kappa'(1)}{\log q} \quad \hbox{and}\ \mu(F) = c_{1} > 0.
\end{equation*}
For each sufficiently small $\delta>0$, there exists a covering  $C$ of $F$ by the arcs~$B(x_i,r_i)$ with centers~$x_i$ and radii~$r_i$ such that~$r_{i} < \delta$ and $\sum_i r_{i}^{\beta_{1}} = c_{2} < \infty$. For $j=1,2 \dots$ let 
\begin{equation*}
C_{j} = \big\{B(x_{i}, r_{i}) \in C\colon q^{-j} \leq r_{i} < q^{-j-1}\big\}.
\end{equation*}
We have
\begin{equation*}
\sum r^{\beta_{1}}_{i} \simeq \sum_{j} q^{-j\beta_{1}} \#C_{j},  
\end{equation*}
so, in particular,~$\#C_{j} \lesssim c_{2}q^{j\beta_{1}}$ for all~$j$. 
By the pigeonhole principle, there exists~$N \gtrsim \log\frac{1}{\delta}$ such that
\begin{equation*}
\mu\Big(F \bigcap \Big(\bigcup\limits_{B(x_i,r_i)\in C_{N}}B(x_i,r_i)\Big)\Big) \geq \frac{6}{\pi^{2}} \frac{c_{1}}{N^{2}}.
\end{equation*}
Since any $B(x_i,r_i) \in C_{N}$ can be covered by at most $q+1$ arcs from the collection~$\{x+[-\frac{1}{2q^{N}},\frac{1}{2q^{N}}]\colon x \in T_{N,N}\}$, there exists a covering
\begin{equation*}
\tilde{C}_{N} \subset \Big\{x+\Big[-\frac{1}{2q^{N}},\frac{1}{2q^{N}}\Big]\colon x \in \TT_{N}\Big\}
\end{equation*}
such that $\#\tilde{C}_{N}\leq \# C_{N}$ and
\begin{equation*}
\mu\Big(\cup_{L \in \tilde{C}_{N}} L\Big) \geq \frac{1}{q+1}\mu\Big(F \bigcap \Big(\bigcup\limits_{B(x_i,r_i)\in C_{N}}B(x_i,r_i)\Big)\Big).
\end{equation*}
Let us call $\mathrm{Mid}(\tilde{C}_{N})$ the set of midpoints of arcs from $\tilde{C}_{N}$.
For the previously obtained $N$, we apply \eqref{app} and Lemma \ref{Frostman} with $\beta > \beta_{1}$ and obtain
\begin{multline*}
\frac{6}{\pi^{2}} \frac{c_{1}}{N^{2}(q+1)} \leq \mu(\cup_{L \in \tilde{C}_{N}} L) \leq \frac{1}{q^{N}}\sum_{x \in  {\mathrm{Mid}(\tilde{C}_{N})}} f(x) \lesssim \big{(} \# C_{N} \ q^{-\beta N} \big{)}^{\gamma}\|\mu\|
\lesssim\\ c_{2}^{\gamma} q^{\gamma(\beta_{1}-\beta)N}.
\end{multline*}
Hence we have $N^{2}q^{-c_{3}N} \geq c_{4} > 0$ for some positive constants $c_{3}, c_{4}$, independent of $\delta$ and $N$. On the other hand, we have $N\to \infty$ when $\delta \to 0$, which leads to a contradiction.
\end{proof}

\section{Proof of Theorem \ref{UP}}
\begin{proof}[Proof of Theorem~\ref{UP}]
In view of Theorem~\ref{transfer}, it suffices to show the inequality
\begin{equation*}
\kappa'(1) \geq -\log |H|
\end{equation*} 
provided~$B \subset H\setminus\{0\}$ and~$\kappa'(1) > -\log |H|$ in the case where the latter inclusion is proper. We will show that
\begin{equation}\label{LpEntropy}
\kappa\Big(\frac{1}{p}\Big) \leq \frac{p-1}{p}\log|H|
\end{equation}
for any~$p \in(1,\infty)$ and this inequality is strict if~$B \ne H$. Until the end of the proof the Fourier transform means the Fourier transform on~$\Z_q$.

Let $v \in W_B$. Then,~$v$ is the $\Z_{q}$-Fourier transform of a vector supported on $H$, so $v=v*\check{\chi}_{H}=v*\chi_{H^{\perp}}$ (provided we properly adjust the constants in the definition of the Fourier transform). Hence, in the coordinates $(h,h') \in H \times H^{\perp} \simeq \Z_{q}$ we have $v(h,h')=v(h,0)$ for all~$h$ and~$h'$ in~$\mathbb{Z}_q$, i.e. $v$ depends on the first coordinate only. We see that each extremal point $x_{0}$ of the set
\begin{equation}\label{addcond}
\Big\{x\in\R_{0}^{q}\colon \forall (h,h')\in \Z_{q} \quad x(h,h')=x(h,0);\ x(h,h')\geq -1\Big\}
\end{equation}
is characterized by the property that the function~$H\ni h \mapsto x_{0}(h,0)$ attains the value $|H|-1$ at some $h$ and $-1$ at the remaining $|H|-1$ elements. From this, the convexity of the~$p$-norm, and formula~\eqref{kappa}, we get
\begin{equation*}
\kappa\Big(\frac{1}{p}\Big) \leq \frac{1}{p}\log\Big(\frac{|H^\perp|}{q}|H|^p\Big) = \frac{p-1}{p}\log|H|.
\end{equation*}

This and the strict convexity of the~$L_p$-norm proves that~\eqref{LpEntropy} is strict provided the inclusion~$B \subset H\setminus\{0\}$ is proper. In this case,~$\kappa'(1) > -\log|H|$ since the function~$\kappa$ is convex.
\end{proof}
\begin{rem}
Theorem \ref{UP} is not true if we consider all complex measures; the counterexample is $B= \{l\}$ and $ \mu = \frac{1}{q}\sum_{k=0}^{q-1}\omega^{kl}\delta_{\{\omega^{k}\}}$.
\end{rem}

%
%
\section{Proof of Theorem \ref{rieszestimate1}}
We will rely upon the simple observation that~$\mu_{a,q}\in \M_{\CC_{\{1,q-1\}}}$. So, our aim is to compute the value~$\kappa'(1)$ for the case~$B = \{1,q-1\}$. In this case, any $v \in W_{B}$ is of the form
\begin{equation*}
v = a \omega_{1} + \bar{a}\omega_{q-1}, \quad \text{for some} \quad a \in \C.
\end{equation*}
The above gives
\begin{equation*} 
W_B = \bigg\{c\Big(\cos\Big(\frac{2\pi j}{q} + \varphi\Big)\Big)_{j=0}^{q-1}\colon c\in \R, \varphi \in [-\pi,\pi]\bigg\}.
\end{equation*}
According to~\eqref{kappaprim}, our purpose is to maximize the quantity
\begin{equation}\label{BigBeautifulFormula}
\sum\limits_{j=0}^{q-1}\Big(1-\gamma\cos\Big(\frac{2\pi j}{q} + \varphi\Big)\Big)\log\Big(1-\gamma\cos\Big(\frac{2\pi j}{q} + \varphi\Big)\Big),
\end{equation}
where~$\gamma$ is chosen in such a way that all the summands are well-defined (the quantity we compute the logarithm of is non-negative) and~$\varphi \in [-\frac{\pi}{q},\frac{\pi}{q}]$. So, we need to maximize a convex function over a convex region. Without loss of generality, we may assume that at least one of the summands vanishes. Since~$\varphi \in [-\frac{\pi}{q},\frac{\pi}{q}]$ this leads to~$\gamma = (\cos \varphi)^{-1}$.

Therefore, the supremum of~\eqref{BigBeautifulFormula} equals
\begin{multline}\label{EntropySpecification}
\sup\limits_{\varphi\in [-\frac{\pi}{q},\frac{\pi}{q}]} \sum\limits_{j=0}^{q-1}\Big(1- \frac{\cos(\frac{2\pi j}{q} + \varphi)}{\cos\varphi}\Big)\log \Big(1- \frac{\cos(\frac{2\pi j}{q} + \varphi)}{\cos\varphi}\Big)=\\
\sup\limits_{\varphi\in [-\frac{\pi}{q},\frac{\pi}{q}]} \sum\limits_{j=0}^{q-1} \Big(1- \cos\frac{2\pi j}{q} + \sin\frac{2\pi j}{q}\tan \varphi\Big)\log\Big(1- \cos\frac{2\pi j}{q} + \sin\frac{2\pi j}{q}\tan \varphi\Big).
\end{multline}
Consider the function~$g$:
\begin{equation*}
g(x) = \sum\limits_{j=0}^{q-1}(a_j + b_j x)\log(a_j + b_j x), \quad x\in \big[-\tan\frac{\pi}{q},\tan\frac{\pi}{q}\big]
\end{equation*}
where~$a_j = 1-\cos\frac{2\pi j}{q}$ and~$b_j = \sin\frac{2\pi j}{q}$. 
\begin{lem}\label{ConvexOptimisationLemma}
For any~$q \geq 3$\textup,
\begin{equation*}
\sup\limits_{x\in [-\tan \frac{\pi}{q},\tan \frac{\pi}{q}]} g(x) = g\Big(\tan \frac{\pi}{q}\Big)
\end{equation*}
\end{lem}
In particular, the supremum in~\eqref{EntropySpecification} is attained at the endpoints since~$\tan$ is a monotone function on~$[-\tan \frac{\pi}{q},\tan \frac{\pi}{q}]$.
\begin{proof}[Proof of Lemma~\ref{ConvexOptimisationLemma}]
Note that~$g$ is convex since the expressions~$a_j + b_j x$ are linear and non-negative when~$x\in [-\tan \frac{\pi}{q},\tan \frac{\pi}{q}]$, and the function~$t\mapsto t\log t$ is convex on the positive semi-axis. It remains to add that~$g$ is symmetric. 
\end{proof}
\begin{proof}[Proof of Theorem~\ref{rieszestimate1}]
The result follows from Theorem~\ref{transfer} and the already proved formula
\begin{equation}\label{KappaPrimRiesz}
\kappa'(1) = -\frac{1}{q}\sum\limits_{j=1}^{q-2}\Big(1 - \frac{\cos\frac{(2j+1)\pi}{q}}{\cos\frac{\pi}{q}}\Big)\log \Big(1 - \frac{\cos\frac{(2j+1)\pi}{q}}{\cos\frac{\pi}{q}}\Big)
\end{equation}
for the case~$B = \{1,q-1\}$. 
\end{proof}

\section{Proof of Proposition~\ref{rieszestimate}}
Assume~$q$ is even. This assumption is pivotal for the forthcoming lemma. 
\begin{lem}\label{Computation}
For any even~$q$\textup, the following identity holds\textup:
\begin{multline}\label{IdentitySumIntegral}
\sum\limits_{j=1}^{q-2}\Big(1 - \frac{\cos\frac{(2j+1)\pi}{q}}{\cos\frac{\pi}{q}}\Big)\log \Big(1 - \frac{\cos\frac{(2j+1)\pi}{q}}{\cos\frac{\pi}{q}}\Big) =\\ (1 - \log 2)q + 2\log 2+ \frac{2}{q\cos\frac{\pi}{q}}\int\limits_{\frac{\pi}{2}}^{\frac{q\pi}{4}}\log(\cos^2z)\sin \frac{2z}{q}\,dz - q\log\cos\frac{\pi}{q}.
\end{multline}
\end{lem} 
\begin{proof}
Consider the function~$f\colon \mathbb{R}\to\mathbb{R}$ defined as follows:
\begin{equation*}
f(a) = \sum\limits_{j=0}^{q-1}\Big(a - \cos\frac{(2j+1)\pi}{q}\Big)\log \Big|a - \cos\frac{(2j+1)\pi}{q}\Big|.
\end{equation*}
The sum on the left hand-side of~\eqref{IdentitySumIntegral} is then equal to
\begin{equation*}
\frac{f(\cos \frac{\pi}{q})}{\cos \frac{\pi}{q}} - q\log \cos \frac{\pi}{q}.
\end{equation*}
The function~$f$ is absolutely continuous and
\begin{equation*}
f'(a) = \log\prod\limits_{j=0}^{q-1}\Big|a - \cos\frac{(2j+1)\pi}{q}\Big|+q = \log\Big(2^{-q+2}T_p^2(a)\Big)+q,
\end{equation*}
where~$q=2p$, by our assumptions,~$p \in\mathbb{N}$, and~$T_p$ is the Chebyshev polynomial of order~$p$, that is
\begin{equation*}
T_p(x) = \cos(p\arccos x) = 2^{p-1}\prod_{j=0}^{p-1}\Big(x-\cos\Big(\frac{(j+\frac12)\pi}{p}\Big)\Big),\qquad x\in [-1,1].
\end{equation*}

Note that by symmetry (here we heavily use that~$q$ is even),~$f(0) = 0$. Thus, since~$f$ is an absolutely continuous function,
\begin{multline*}
f\Big(\cos\frac{\pi}{q}\Big) = \int\limits_0^{\cos \frac{\pi}{q}} \Big(\log\Big(2^{-q+2}T_p^2(a)\Big)+q\Big)\,da =\\ (1-\log 2)q\cos\frac{\pi}{q} + 2\log 2\cos\frac{\pi}{q}+ \int\limits_0^{\cos\frac{\pi}{q}}\log\cos^2(p\arccos a)\,da = \\
(1-\log 2)q\cos\frac{\pi}{q} + 2\log 2\cos\frac{\pi}{q}+ \int\limits_{\frac{\pi}{q}}^{\frac{\pi}{2}} \log\cos^2(px)\sin x\,dx =\\ (1-\log 2)q\cos\frac{\pi}{q} + 2\log 2\cos\frac{\pi}{q} + \frac{2}{q}\int\limits_{\frac{\pi}{2}}^{\frac{q\pi}{4}}\log\cos^2z\sin\frac{2z}{q}\,dz.
\end{multline*}
So, the sum on the left hand-side of~\eqref{IdentitySumIntegral} equals
\begin{equation*}
(1-\log 2)q + 2\log 2 + \frac{2}{q\cos\frac{\pi}{q}}\int\limits_{\frac{\pi}{2}}^{\frac{q\pi}{4}}\log\cos^2z\sin\frac{2z}{q}\,dz - q\log\cos\frac{\pi}{q}.
\end{equation*}
\end{proof}
\begin{proof}[Proof of Proposition~\ref{rieszestimate}]
Since~$\mu_{a,q}\in \M_{\CC_{\{1,q-1\}}}$, Theorem~\ref{transfer} says that
\begin{equation*}
\dH(\mu_{a,q}) \geq 1+\frac{\kappa'(1)}{\log q}.
\end{equation*}
Thus, it remains to combine this estimate with formula~\eqref{KappaPrimRiesz} and Lemma~\ref{Computation}.
\end{proof}

\section{Proof of Proposition \ref{RieszEstimate}}
Let us begin with two technical lemmas.

\begin{lem} \label{numerical}
Let~$a\in [-1,1]$ be fixed. Consider the function~$g\colon\mathbb{R} \to \mathbb{R}$ given by the formula 
\begin{equation*}
g(x)=(1+a\cos x)\log(1+a\cos x).
\end{equation*}
Then\textup, for any~$a\in [-1,1]$ and any~$x$\textup,
\begin{equation*}
 |g'(x)|\leq 2.
\end{equation*}
\end{lem}
\begin{proof}
We have
\begin{equation*}
|g'(x)| = |-a\sin x - a\sin x\log(a\cos x + 1)|.
\end{equation*}
We fix~$x$ for a while and treat this expression as a function of~$a$. This function is convex (since the function~$t\log t$ is), so it attains its maximum at the endpoints~$a=\pm 1$. Therefore, it suffices to prove the inequality
\begin{equation*}
\sin x\big(1+\log(1+\cos x)\big) \leq 2,\quad x\in \big[0,\frac{\pi}{2}\big].
\end{equation*}
We estimate~$\sin x$ by one,~$1+\cos x$ by~$e$, and get it.
\end{proof}
\begin{rem}
If we denote
\begin{equation*}
L = \sup_{x\in [0,\frac{\pi}{2}]}\sin x\big(1+\log(1+\cos x)\big),
\end{equation*}
then the constant~$2$ in Lemma~\ref{numerical} might be replaced by~$L$. The numerical bound for~$L$ is~$1.25$.
\end{rem}

The following lemma may be found in the proof of Theorem 1 in \cite{F2}:

\begin{lem}\textup{(\cite{F2})}\label{1lip}
	For $a\in [-1,1]$\textup, let
	\begin{equation*}
	h(a) = \int_{0}^{2\pi} (1+a\cos x)\log(1+a\cos x)\frac{dx}{2\pi}
	\end{equation*}
	Then $h$ is $1$-Lipschitz.
\end{lem}
\begin{proof}
	Similar to the proof of Lemma~\ref{numerical}, the function~$h$ is convex (since the function~$t\log t$ is). Moreover,~$h'$ is an odd function. Thus,~$h'(a) \leq h'(1) = 1$.
\end{proof}

\begin{proof}[Proof of Proposition~\ref{RieszEstimate}]
By Theorem~\ref{transfer},
\begin{equation*}
\dH(\mu_{a,q}) \geq 1+\frac{\kappa'(1)}{\log q},
\end{equation*}
where~$\kappa'(1)$ is defined in~\eqref{KappaPrimRiesz}. By Lemma~\ref{numerical},
\begin{equation*}
\kappa'(1) \geq -\frac{1}{2\pi}\int\limits_0^{2\pi} \Big(1+\frac{\cos x}{\cos\frac{\pi}{q}}\Big)\log\Big(1+\frac{\cos x}{\cos\frac{\pi}{q}}\Big)\,dx - \frac{4\pi}{q}.
\end{equation*}
It remains to apply Lemma~\ref{1lip} and use the identity
\begin{equation*}
\frac{1}{2\pi}\int\limits_0^{2\pi}(1+\cos x)\log(1+\cos x)\,dx = 1-\log 2.
\end{equation*}
 \end{proof}
 
\section{Further examples and comments}
A more general form of the Gundy--Varopoulos martingale also appears as an element of the proof of the dimension estimate in \cite{Pey}. In that paper, it is used to prove a version of the pointwise ergodic theorem with respect to Riesz products. 

The assumption of being a non-negative measure from $\M_{B}(\T)$ implies the symmetry of~$B$. Theorems corresponding to the case when $B$ is (strongly) antisymmetric were considered in \cite{CJ}.

\begin{rem} For a fixed $q$\textup, if $B\neq \Z_{q}\setminus\{0\}$, then
\[
	\dH(\mu) \geq \delta_{q} > 0	
\]
for any finite non-negative measure from $\M_{C_{B}}(\T)$. If $q$ is small\textup, then the constant $\delta_{q}$ may be estimated by the analysis of extremal points of
\[
span\{\omega_{m}\}_{m \in B} \cap \{x \in \R_{0}^{q}: \forall j \quad x_{j} \geq -1\}.
\]
For example\textup, if $q=4$ then we may take $B=\{2\}$ or $B=\{1,3\}$. In the first case, the extremal points are $\pm (1,-1,1,-1)$\textup, while for the second choice they are $\pm(1,1,-1,-1),\pm (1,-1,-1,1)$. This gives $\delta_{4} \geq \frac{1}{2}$.
\end{rem}
An obvious converse of Theorem~\ref{UP} says that singular measures have rich spectrum in the arithmetical sense.
\begin{cor}
	Let $\mu \in \M(\T)$ be a non-negative finite measure such that
	\[
	\dH(\mu) < \delta_{q},
	\]
	where $\delta_{q}$ is as in the above remark. Then for each $ m \in \{1,\dots, q-1\}$  there exists $n \in \mathrm{spec}(\mu)$ such that $n$ has a divisor with residue $m$ modulo $q$.
\end{cor}

\end{document}